\documentclass[amsart]{amsart}
\usepackage{amsmath,amssymb,amsthm}
\usepackage{mathrsfs}
\usepackage{graphicx}
\usepackage{tikz-cd}
\usepackage{enumerate}
\usepackage{cite}
\usepackage{url}
\usepackage{amsfonts}
\usepackage{latexsym}
\usepackage{enumitem}
\usepackage{mathrsfs}
\usepackage[all]{xy} \SelectTips{eu}{}
\usepackage{enumerate}
\usepackage{hyperref}
\usepackage{color}

\numberwithin{equation}{section}
\newtheorem{theorem}{Theorem}[section]
\newtheorem{lemma}[theorem]{Lemma}
\newtheorem{proposition}[theorem]{Proposition}
\newtheorem{corollary}[theorem]{Corollary}
\newtheorem{definition}[theorem]{Definition}
\newtheorem{remark}[theorem]{Remark}

\newcommand{\Gen}{\operatorname{Gen}}
\newcommand{\Cogen}{\operatorname{Cogen}}
\newcommand{\Coker}{\operatorname{Coker}}
\newcommand{\Tor}{\operatorname{Tor}}
\newcommand{\Ext}{\operatorname{Ext}}
\newcommand{\Hom}{\operatorname{Hom}}

\newcommand{\id}{\operatorname{id}}
\newcommand{\pd}{\operatorname{pd}}

\newcommand{\Mod}{\operatorname{Mod}}
\newcommand{\Ker}{\operatorname{Ker}}

\newcommand{\A}{\mathcal{A}}
\newcommand{\B}{\mathcal{B}}
\def\Id{\mathrm{Id}}
\newcommand{\Proj}{\operatorname{Proj}}

\begin{document}

\title{Silting and tilting objects in cleft extensions of abelian categories}
\author{Guoqiang Zhao}
\address{Department of Mathmatics, Hangzhou Dianzi University, Hangzhou 310018, China}
\email{gqzhao@hdu.edu.cn}

\author{Juxiang Sun}
\address{School of Mathematics and Statistics, Shangqiu Normal University, Shangqiu 476000, China}
\email{Sunjx8078@163.com}

\begin{abstract} 
We establish connections between silting and tilting objects in an abelian category $\mathcal{B}$
and those in a cleft extension $\mathcal{A}$ of $\mathcal{B}$, which provides a method for constructing more silting and tilting objects.
Then we apply our results to the cleft extensions of module categories,  and characterize silting and tilting modules over  $\theta$-extension of rings. Some known results over trivial extension of rings are extended and strengthened.
\end{abstract}

\maketitle



\section{Introduction}
The study of cleft extension of rings is quite essential and has interesting applications to classical ring theory, since this class of ring extensions unifies large parts of rings, such as trivial extension rings and in particular triangular matrix rings, tensor rings and positively graded rings and so on. 
The concept of cleft extension of abelian categories was
introduced by Beligiannis \cite{Bel00} as a natural generalization of cleft extension of rings and
trivial extension of abelian categories defined by Fossum-Griffith-Reiten \cite{FGR75}.
The cleft extension of abelian category provide a general setting and unified treatment for the study of cleft extension of rings.
Moreover, the framework of cleft extension was used to reduce some homological properties of rings, such as the finitistic dimensions \cite{EGPS22, GPS21}, cotorsion pairs \cite{H25},
injective generations \cite{KP25}, the Igusa-Todorov distances and
extension dimensions \cite{MZL25}, the Gorenstein weak global (flat-cotorsion) dimensions
and relative singularity categories \cite{LMY25},
the Gorenstein projective modules, singularity categories and Gorensteinness \cite{EPS22, K26}.

On the other hand, the silting theory, originated from the work of Keller and Vossieck \cite{KV88}, 
seems to be the tilting theory in the level of derived categories. 
Silting complexes\cite{KV88, W13} and cosilting complexes\cite{ZW17}, can be understood as derived analogues of tilting and cotilting modules. 
In \cite{AMV16}, the notion of silting modules was introduced by Angeleri H\"ugel, Marks and Vit\'oria, 
as a common generalization of 1-tilting modules over an arbitrary ring and $\tau$-tilting modules over a finite dimensional algebra due to Adachi, Iyama and Reiten \cite{AIR14}. 
It was shown that silting modules are in bijection with 2-term silting complexes and with certain t-structures and co-t-structures in the derived module category.
Dually, the notion of cosilting modules, was introduced by Breaz, Pop \cite{BP17}, Zhang and Wei \cite{ZW17}. 
Many important properties of these modules were presented in the past few years (see \cite{A19, AH17, AMV17, BZ18, CTT97, MV18, P17}). 
Subjects related to transfer of silting modules were also studied.
For instance, silting modules have been studied over formal triangular matrix rings \cite{GH20}, over trivial ring extensions\cite{M22}, and faithfully flat ring extensions\cite{AC22,B20}. 

The main aim of the present paper is to investigate silting and $n$-tilting objects over cleft extensions of abelian categories. The paper is organized as follows:

In Section 2, we recall some basic concepts and facts.

Suppose that an abelian category $\A$ is a cleft extension of an abelian category $\B$. 
In Section 3, we study how to construct (partial) silting and tilting objects of $\A$ from those objects of $\B$, respectively (see Theorems 3.3 and 3.5). 
The transfers from (partial) silting objects of $\A$ to $\B$ are also discussed (see Theorem 3.8).

Section 4 is devoted to applications in cleft extension of module categories. In particular,
We obtain characterizations of (partial) silting and tilting modules over $\theta$-extensions and tensor rings. Surprisingly, when the $\theta$-extension is induced by a finite dimensional algebra and a finitely generated bimodule, we can obtain fruitful results, especially for $\tau$-rigid and support $\tau$-tilting modules. Some known results over trivial ring extensions, including triangular matrix rings and Morita context rings, are extended.

\section{Preliminaries and notations}

Throughout this paper, we always assume that $\A$ and $\B$ are abelian categories with enough projective and injective objects. For an object $X$ in $\A$, denote by $\pd_{\A}(X)$ and $\id_{\A}(X)$ the projective and injective dimensions of $X$, respectively.

Let $X\xrightarrow{\sigma}Y$ be a morphism in $\A$. Write
\begin{gather*}
\mathcal{D}_{\sigma} =\{L\in\A: \Hom_{\A}(Y,L)\xrightarrow{\Hom_{\A}(\sigma,L)}\Hom_{\A}(X,L)\text{ is an epimorphism}\};\\
\mathcal{C}_{\sigma} =\{L\in \A: \Hom_{\A}(L,X)\xrightarrow{\Hom_{\A}(L,\sigma)}\Hom_{\A}(L,Y)\text{ is an epimorphism}\};\\
\end{gather*}

We remark that if $\sigma$ is a map in $\Proj(\A)$, it is standard to verify that $\mathcal{D}_{\sigma}$ is closed under epimorphic images, extensions and product.

Given $X\in\A$,  we denote by $\Gen(X)$ the class consisting of quotients of coproducts of copies of $X$ and $\Cogen(X)$ the class consisting of subobjects of products of copies of $X$. Write $X^{\perp_{n}}=\{Y\in\A: \Ext^{j}_{\A}(X, Y)=0$ for all $1\leqslant j\leqslant n \}$, and
$X^{\perp}=\{Y\in\A: \Ext^{j}_{\A}(X, Y)=0$ for all $j\geqslant 1\}$,

Let $P_{1}\xrightarrow{\sigma}P_{0}\to X\to 0$ be a projective presentation of $X$. Analogous to the definition of (partial) silting modules in \cite{AMV16}, $X$ is called a {\it partial silting object} with respect to $\sigma$ if $X\in\mathcal{D}_{\sigma}$ and $\mathcal{D}_{\sigma}$ is a torsion class (equivalently, $\mathcal{D}_{\sigma}$ is closed under coproduct by a similar argument to \cite[Remark 3.11]{AMV16}). 
$X$ is called a {\it silting object} with respect to $\sigma$ if $\Gen(X)=\mathcal{D}_{\sigma}$.

Dually, Let $0\to X\to E^{0}\xrightarrow{\xi}E^{1}$ be an injective copresentation of $X$. $X$ is called partial cosilting with respect to $\xi$ if $X\in\mathcal{C}_{\xi}$ and $\mathcal{C}_{\xi}$ is a torsionfree class. $X$ is called cosilting with respect to $\xi$ if $\Cogen(X)=\mathcal{C}_{\xi}$.

We first present some properties of adjoint pairs.

\begin{lemma} \label{lem:2.1}
Suppose that the functor $\mathbb{G}: \A\rightarrow\B$ has a left adjoint functor $\mathbb{F}$. 
\begin{enumerate}
\item Given $\sigma\in\Hom_{\B}(B_{1}, B_{2})$, then , for every $A\in\mathcal{A}$,
$\Hom_{\A}(\mathbb{F}(\sigma), A)$ is epic (resp. monic)
if and only if $\Hom_{\B}(\sigma$, $\mathbb{G}(A))$ is epic (resp. monic). 

\item Let $\nu\colon\mathsf{Id}_{\mathcal{B}}\rightarrow\mathbb{GF}$ be the unit and 
$\mu\colon\mathbb{FG}\rightarrow \Id_{\mathcal{A}}$ the counit of the adjoint pair $(\mathbb{F}, \mathbb{G})$.
If $\mathbb{G}$ is faithful exact, then, for every $A\in\mathcal{A}$,
$\mu_{A}\colon \mathbb{FG}(A)\rightarrow A$ is an epimorphism. 
\end{enumerate}
\end{lemma}

\begin{proof}

(1) Since $(\mathbb{F}, \mathbb{G})$ is an adjoint pair, we get the following commutative diagram
\[
\begin{array}{ccc}
\Hom_{\A}(\mathbb{F}(B_{2}), A) & \xrightarrow{\Hom_{\A}(\mathbb{F}(\sigma), A)} & \operatorname{Hom}_{\A}(\mathbb{F}(B_{1}), A) \\
\cong \Bigg\downarrow & & \Bigg\downarrow \cong \\
\operatorname{Hom}_{\B}(B_{2}, \mathbb{G}(A)) & \xrightarrow{\Hom_{\B}(\sigma, \mathbb{G}(A))} & \operatorname{Hom}_{\B}(B_{1}, \mathbb{G}(A))
\end{array}
\]
So we obtain the assertion.

(2) For every $A\in\mathcal{A}$ and $B\in\mathcal{B}$, the following relations are satisfied: 
 \[
 \mathsf{Id}_{\mathbb{F}(B)}=\mu_{\mathbb{F}(B)}\mathbb{F}(\nu_B) \  \text{ and } \  \mathsf{Id}_{\mathbb{G}(A)}=\mathbb{G}(\mu_A)\nu_{\mathbb{G}(A)}.
 \]
which implies that $\mathbb{G}(\mu_{A})$ is a split epimorphism.
Since $\mathbb{G}$ is faithful exact, it follows that $\mu_{A}\colon \mathbb{FG}(A)\rightarrow A$ is an epimorphism for every $A\in\mathcal{A}$.
\end{proof}

\begin{definition}\textnormal{(\!\!\!\cite{Bel00})}
{\rm 
A {\it cleft extension} of an abelian category $\mathcal{B}$
is an abelian category $\mathcal{A}$
together with functors $$\xymatrix@!=4pc{ \mathcal{B} \ar[r]|{i} & \mathcal{A}
			 \ar[r]|{e} & \mathcal{B}
			\ar@/_1pc/[l]|{l} }  $$
such that the functor $e$ is faithful exact and admits a left adjoint $l$, and there is a natural isomorphism
$ei \simeq \Id_{\mathcal{B}}$.
}
\end{definition}

From now on we will denote a cleft extension by $(\mathcal{B},\mathcal{A},i,e,l)$, and it
gives rise to additional properties. For instance, it follows that the functor $i$ is fully
faithful and exact (see \cite[Lemma 2.2]{Bel00}, \cite[Lemma 2.2(ii)]{GPS21}).  Moreover, there is a
functor $q: \mathcal{A} \rightarrow \mathcal{B}$, which is left adjoint of $i$ (see \cite[Proposition 2.3]{Bel00}
and \cite[Lemma 2.2(iv)]{GPS21}). Then, $(ql, ei)$ is an adjoint pair and since $ei \simeq \Id_{\mathcal{B}}$,
it follows that $ql \simeq \Id_{\mathcal{B}}$. We mention that the functors $l$ and $q$ are both right exact and preserve coproduct.

Denote by $\nu\colon\mathsf{Id}_{\mathcal{B}}\rightarrow\mathsf{el}$ the unit and by $\mu\colon\mathsf{le}\rightarrow \mathsf{Id}_{\mathcal{A}}$ the counit of the adjoint pair $(\mathsf{l},\mathsf{e})$. 
It follows from Lemma 2.1(2) that 
there is an exact sequence of functors
$$0\rightarrow G\rightarrow le\xrightarrow{\mu} \Id_{\mathcal{A}}\rightarrow 0\qquad(2.1)$$
where $G$ is an endofunctor of $\A$ defined by $G(A) = \ker \mu_{A}$.

Because $e(\mu_{i(B)})$ is a split epimorphism for every $B\in\mathcal{B}$ and $ei \cong \Id_{\mathcal{B}}$, one has a split exact sequence
$$0\rightarrow F\rightarrow el\rightarrow \Id_{\mathcal{B}}\rightarrow 0,\qquad(2.2)$$
where $F= eGi$ is an endofunctor of $\B$ ( for details see \cite[Section 2]{GPS21}).


In the theory of trivial extensions of abelian categories, there is the dual notion of trivial coextensions (or ``left" trivial extensions). Likewise, there is a dual notion to cleft extensions which we recall below. 

\begin{definition}\textnormal{(\!\!\!\cite{Bel00})}
{\rm
A {\it cleft coextension} of an abelian category $\mathcal{B}$
is an abelian category $\mathcal{A}$
together with functors $$\xymatrix@!=4pc{ \mathcal{B} \ar[r]|{i} & \mathcal{A}
			 \ar[r]|{e} & \mathcal{B}
			\ar@/^1.8pc/[l]_-{r} }  $$
such that the functor $e$ is faithful exact and admits a right adjoint $r$, and there is a natural isomorphism
$ei \simeq \Id_{\mathcal{B}}$.
}
\end{definition}

From now on we will denote a cleft coextension by $(\mathcal{B},\mathcal{A}, i, e, r)$. Like with cleft extensions, there is more structure that arises from the above information. For example, it turns out that $i$ is fully faithful exact and that there is a functor $p\colon\mathcal{A}\rightarrow \mathcal{B}$ such that $(i, p)$ is an adjoint pair. Necessarily we have that $pr\simeq\mathsf{Id}_{\mathcal{B}}$. There are endofunctors $\mathsf{G}'\colon\mathcal{A}\rightarrow\mathcal{A}$ and $\mathsf{F}'\colon\mathcal{B}\rightarrow\mathcal{B}$ that appear in the following short exact sequences: 
\[
0\rightarrow \mathsf{Id}_{\mathcal{A}}\rightarrow \mathsf{re}\rightarrow\mathsf{G}'\rightarrow 0 \  \text{ and }  \ 0\rightarrow \mathsf{Id}_{\mathcal{B}}\rightarrow \mathsf{er}\rightarrow \mathsf{F'}\rightarrow 0.
\]
The second one splits. 

All the results about cleft extensions have a dual translation to results about cleft coextensions. We study only cleft extensions.

\section{Silting objects in cleft extensions of abelian categories}

\subsection{Lifting silting and tilting objects of $\B$ to $\A$}


\begin{lemma} \label{lem:3.1} \textnormal{(\!\!\!\cite[Lemma 2.3]{K26})}
Let $(\mathcal{B},\mathcal{A}, i,e,l)$ be a cleft extension.
$A\in\A$ is projective if and only if $A$ is a direct summand of $l(P)$ for some $P\in\Proj(\B)$.
\end{lemma}

\begin{lemma} \label{lem:3.2}
Let $(\mathcal{B},\mathcal{A},i,e,l)$ be a cleft extension,
and $B_{1}\xrightarrow{\sigma}B_{2}$ be a morphism in $\mathcal{B}$. 
For $A\in\A$, the following hold:
\begin{enumerate}
\item $A\in\mathcal{D}_{l(\sigma)}$ if and only if $e(A)\in\mathcal{D}_{\sigma}$.

\item If $i$ preserves coproduct and $\mathcal{D}_{l(\sigma)}$ is closed under coproduct, then $\mathcal{D}_{\sigma}$ is so;
If $e$ preserves coproduct and $\mathcal{D}_{\sigma}$ is closed under coproduct, then $\mathcal{D}_{l(\sigma)}$ is so.
\end{enumerate}
\end{lemma}

\begin{proof}
(1) Since $(l, e)$ is an adjoint pair, it follows that
$\Hom_{\A}(l(\sigma)$, $A)$ is epic
if and only if $\Hom_{\B}(\sigma$, $e(A))$ is epic by Lemma \ref{lem:2.1}(1).
So $A\in\mathcal{D}_{l(\sigma)}$ if and only if $e(A)\in\mathcal{D}_{\sigma}$.

(2) ``$\Rightarrow$'' Let $(B_{j})_{j\in J}$ be a family of objects in $\B$ with $B_{j}\in\mathcal{D}_{\sigma}$. Since $ei(B_{j})\cong B_{j}$,
then $i(B_{j})\in\mathcal{D}_{l(\sigma)}$ by (1). 
Since $i$ preserves coproduct, then 
$i(\coprod_{j\in J}B_{j})$ $\cong$ $\coprod_{j\in J}i(B_{j})\in\mathcal{D}_{l(\sigma)}$ by assumption. 
Thus $\coprod_{j\in J}B_{i}\in\mathcal{D}_{\sigma}$ by (1).

``$\Leftarrow$'' Let $(A_{j})_{j\in J}$ be a family of objects in $\A$ with $A_{j}\in\mathcal{D}_{l(\sigma)}$. Then $e(A_{j})\in\mathcal{D}_{\sigma}$ by (1). 
Thus $e(\coprod_{j\in J}A_{j})\cong\coprod_{j\in J}e(A_{j})\in\mathcal{D}_{\sigma}$. Then $\coprod_{j\in J}A_{j}\in\mathcal{D}_{l(\sigma)}$ by (1) again.
\end{proof}

Let $P_{1}\xrightarrow{\sigma}P_{0}\to B\to 0 (\ast)$ be a projective presentation of $B$ in $\B$. 
Note that $l$ is right exact, one has a projective presentation 
$l(P_{1})\xrightarrow{l(\sigma)} l(P_{0})\to l(B)\to 0$ by \cite[Lemma 2.2(i)]{K26}. We can now formulate the following results.

\begin{theorem} \label{thm:3.3}
Let $(\mathcal{B},\mathcal{A},i,e,l)$ be a cleft extension with $i$ and $e$ preserving coproduct, and
 $P_{1}\xrightarrow{\sigma}P_{0}\to B\to 0$ be a projective presentation of $B$ in $\B$.
\begin{enumerate}
\item $l(B)$ is a partial silting object in $\A$ with respect to $l(\sigma)$ if and only if $B$ is a partial silting object in $\B$ with respect to $\sigma$ and $F(B)\in\mathcal{D}_{\sigma}$.
\item $l(B)$ is a silting object in $\A$ with respect to $l(\sigma)$ if and only if 
$B$ is a silting object in $\B$ with respect to $\sigma$ and $F(B)\in\Gen(B)$.
\end{enumerate}
\end{theorem}

\begin{proof}
(1)  If $l(B)\in\mathcal{D}_{l(\sigma)}$, 
one has that $el(B)\in\mathcal{D}_{\sigma}$ by Lemma \ref{lem:3.2}(1).
Note that $el\simeq F\oplus \Id_{\B}$,
it follows that $B$ and $F(B)$ are in $\mathcal{D}_{\sigma}$ as $\mathcal{D}_{\sigma}$ is closed under epimorphic images.

If $B$ and $F(B)$ are in $\mathcal{D}_{\sigma}$, then $el(B)\in\mathcal{D}_{\sigma}$ since $\mathcal{D}_{\sigma}$ is closed under extensions.
Lemma \ref{lem:3.2}(1) infers that $l(B)\in\mathcal{D}_{l(\sigma)}$.

By assumption $e$ and $i$ preserve coproduct, the assertion follows form Lemma \ref{lem:3.2}(2).

(2) ``$\Rightarrow$'' By (1), $B$ is a partial silting object in $\B$ with respect to $\sigma$ and $F(B)\in\mathcal{D}_{\sigma}$. So $\Gen(B)\subseteq\mathcal{D}_{\sigma}$ from the fact that $\mathcal{D}_{\sigma}$ is closed under coproduct and epimorphic images.

Let $B'\in\mathcal{D}_{\sigma}$. Then $ei(B')\cong B'\in\mathcal{D}_{\sigma}$. Lemma \ref{lem:3.2}(1) infers $i(B')\in\mathcal{D}_{l(\sigma)}=\Gen(l(B))$ by assumption.
Thus there are a cardinal $I$ and an epimorphism $f:l(B^{(I)})\to i(B')$ as $l$ preserves coproduct, and so is $q(f): B^{(I)}\cong ql(B^{(I)})\to qi(B')$ for the sake of $q$ being right exact.

Denote by $\xi\colon\mathsf{Id}_{\mathcal{A}}\rightarrow iq$ the unit and by $\eta\colon qi\rightarrow \mathsf{Id}_{\mathcal{B}}$ the counit of the adjoint pair $(q,i)$. 
Note that $i$ is faithful exact, we infer that $qi(B)\overset{\eta_{B}}\to B\to 0$ is exact for all $B\in\B$ by Lemma \ref{lem:2.1}(2).
So we get an epimorphism $$\eta_{B'}q(f):B^{(I)}\to B'.$$ 
Hence $B'\in\Gen(B)$, and so $\mathcal{D}_{\sigma} = \Gen(B)$. 
Thus $B$ is a silting object with respect to $\sigma$ and $F(B)\in\Gen(B)$.

``$\Leftarrow$'' Because $B$ is a silting obiect with respect to $\sigma$, it follows that $\Gen(B)=\mathcal{D}_{\sigma}$,
which yields  $F(B)\in\mathcal{D}_{\sigma}$. 
From (1) we know that $l(B)$ is partial silting with respect to $l(\sigma)$.
Thus $\Gen(l(B))\subseteq\mathcal{D}_{l(\sigma)}$.

Let $A$ be in $\mathcal{D}_{l(\sigma)}$. Then $e(A)\in\mathcal{D}_{\sigma}=$ $\Gen(B)$,
and hence there are a cardinal $J$ and an epimorphism $g: B^{(J)}\to e(A)$.
Thus $l(g): l(B^{(J)})\cong (l(B))^{(J)}\to le(A)$ is epic as $l$ is right exact.
It follows from (2.1) that $\mu_{A}\circ l(g): (l(B))^{(J)}\to A$ is surjective,
which implies $A\in\Gen(l(B))$. Thus $\Gen(\l(B))=\mathcal{D}_{l(\sigma)}$, and we complete the proof.
\end{proof}

We call an object $X\in\B$ \emph{partial $n$-tilting} if (T1) $\pd_{\A}X\leq n$, and (T2) $\Ext_{\B}^j(X$, $X)=0$ for any $j\geq 1$, that is $X\in X^{\perp}$.
Moreover, $X$ is called \emph{$n$-tilting} if in addition (T3) for any $P\in\Proj(\B)$, there exists an exact sequence 
$$0\to P\to X_{0}\to\cdots\to X_{n}\to 0 $$ with each $X_{j}\in\mathrm{Add}X$.

By \cite[Proposition 3.13(1)]{AMV16}, $X$ is 1-tilting if and only if $X$
is silting with respect to a monomorphism $\sigma$ , where  $P_{1}\xrightarrow{\sigma}P_{0}\to X\to 0$ is a projective presentation of $X$. 
Next we will discuss whether $l$ lifts general $n$-tilting objects.

\begin{lemma}\label{lem:3.4}  \textnormal{(\!\!\!\cite[Lemma 2.2]{K26})} 
Let $(\mathcal{B},\mathcal{A},i,e,l)$ be a cleft extension and $X\in\B$. 
The following hold: 
    \begin{itemize}
        \item[(1)] $\mathbb{L}_{j}F(X)\cong \mathsf{e}\mathbb{L}_{j}l(X)$ for all $j\geq 1$, where $\mathbb{L}_{j}$ is the $j$th left derived functor. 
        \item[(2)] If $\mathbb{L}_{j}F(X)=0$ for all $j\geq 1$, then $\Ext_{\mathcal{A}}^j(l(X),Y)\cong \mathsf{Ext}_{\mathcal{B}}^j(X,e(Y))$ for all $j\geq 1$ and every object $Y$ of $\mathcal{A}$.
    \end{itemize} 
\end{lemma} 

\begin{theorem} \label{thm:3.5}
Let $(\mathcal{B},\mathcal{A},i,e,l)$ be a cleft extension and $X\in\B$ with $\mathbb{L}_{t}F(X)=0$ for all $1\leq t\leq n+1$. 
The following hold:
 \begin{itemize}
 \item[(1)] $l(X)$ is partial $n$-tilting iff $X$ is partial $n$-tilting and $F(X)\in X^{\bot_{n}}$.
 \item[(2)] If $\Ker q = 0$, then $l(X)$ is $n$-tilting iff $X$ is $n$-tilting and $F(X)\in X^{\bot_{n}}$.
  \end{itemize}
\end{theorem}

\begin{proof}
(1) ``$\Rightarrow$'' Suppose that $l(X)$ is partial $n$-tilting. Then $\pd_{\A}l(X)\leq n$,
and for any $B\in\B$, it follows from Lemma \ref{lem:3.4}(2) that
$$\Ext_{\mathcal{B}}^{n+1}(X, B)\cong\Ext_{\mathcal{B}}^{n+1}(X, ei(B))\cong\Ext_{\mathcal{A}}^{n+1}(\mathsf{l}(X), i(B))=0,$$ which yields $\pd_{\B}X\leq n$. 
From Lemma \ref{lem:3.4}(2) and (2.2) we know that 
for any $1\leq j\leq n$, $0=\Ext_{\A}^j(l(X),l(X))\cong \Ext_{\B}^j(X,el(X))\cong\Ext_{\B}^j(X,X\oplus F(X)).$
Thus $F(X)\in X^{\bot_{n}}$ and $X\in X^{\bot}$, as desired.

``$\Leftarrow$'' Let $X$ be partial $n$-tilting, then $\pd_{\A}X\leq n$.
For any $Y\in\A$, it follows from Lemma \ref{lem:3.4}(2) that
$\Ext_{\mathcal{A}}^{n+1}(l(X), Y)\cong\Ext_{\mathcal{B}}^{n+1}(X, e(Y))=0,$
which yields $\pd_{\A}l(X)\leq n$. From Lemma \ref{lem:3.4}(2) and (2.2) we know that 
for any $1\leq j\leq n$, $\Ext_{\A}^j(l(X),l(X))\cong \Ext_{\B}^j(X,el(X))\cong\Ext_{\B}^j(X,X\oplus F(X))=0$ by assumption.

(2) ``$\Rightarrow$'' By (1) it suffices to show that $X$ satisfies (T3).
Let $P\in\Proj(\B)$, then  $l(P)\in\Proj(\A)$ and since $l(X)$ is $n$-tilting,
there exists an exact sequence $$0\to l(P)\to T_{0}\to\cdots\to T_{n}\to 0 \qquad(\ast)$$ 
with each $T_{t}\in\mathrm{Add}l(X)$.
Note that $ql\simeq\Id_{\B}$, thus $\mathbb{L}_{j}ql=0$, and hence $\mathbb{L}_jq(T_{j})=0$ for any $j\geq 1$.
Applying the functor $q$ to $(\ast)$ gives an exact sequence
$0\to P\to q(T_{0})\to\cdots\to q(T_{n})\to 0$ with each $q(T_{j})\in\mathrm{Add}X$.
Therefore, $X$ is $n$-tilting.

``$\Leftarrow$''
Let $A\in\Proj(\A)$.Lemma 3.1 infers that $A$ is a direct summand of $l(P)$ for some $P\in\Proj(\B)$, and so $q(A)$ is a direct summand of $P$.
Because $X$ is $n$-tilting, there exists an exact sequence 
$$0\to q(A)\to X_{0}\to\cdots\to X_{n}\to 0 \qquad(\ast\ast)$$
with each $X_{t}\in\mathrm{Add}X$.
For each $1\leq j\leq n$, since $e$ is faithful and $\mathbb{L}_j\mathsf{F}(X)=0$, it follows from  Lemma \ref{lem:3.4}(1) that
$\mathbb{L}_{j}l(X) =0$.
Thus by applying the functor $l$ to $(\ast\ast)$ one gets an exact sequence
$$0\to lq(A)\to l(X_{0})\to\cdots\to l(X_{n})\to 0$$ 
with each $l(X_{t})\in\mathrm{Add}l(X)$.
Notice taht $q(A-lq(A))=q(A)-qlq(A)=0$, then $A-lq(A)\in \Ker q$, and hence $A=lq(A)$ by hypothesis. 
Therefore, $l(X)$ satisfies (T3), as required.
\end{proof}


\subsection{Transfers from silting objects of $\A$ to $\B$}

\begin{lemma} \label{lem:3.6}
Let $(\mathcal{B},\mathcal{A}, i,e,l)$ be a cleft extension.
 Suppose that $A_{1}\overset{\delta}{\rightarrow}A_{2}$ is a morphism in $\A$. Then $i(B)\in\mathcal{D}_{\delta}$ if and only if $B\in\mathcal{D}_{\mathbf{q}(\delta)}$.
\end{lemma}

\begin{proof}
Since $(q,i)$ is an adjoint pair,  Lemma \ref{lem:2.1}(1) infers that $\Hom_{\B}(q(\delta), B)$ is epic if and only if $\Hom_{\A}(\delta, i(B))$ is epic.
So $i(B)\in\mathcal{D}_{\delta}$ if and only if $B\in\mathcal{D}_{q(\delta)}$.
\end{proof}

\begin{lemma} \label{lem:3.7}
For every $A\in\A$, it has a projective resolution of the form $\cdots\to l(P_{1})\to l(P_{0})\to A\to 0$ with each $P_{j}\in\Proj(\B)$.
\end{lemma}

\begin{proof}

From \cite[Lemma 2.3]{K26} we know that $Q\in\A$ is projective if and only if it is a direct summand of $l(P)$ for some $P\in\Proj(\B)$. Let $\cdots\to Q_{2}\to Q_{1} \to Q_{0}\to A\to 0$ be a projective resolution of $A\in\A$. Then there is $P_{0}\in\Proj(\B)$ such that $Q_{0}\oplus Q'_{0}=l(P_{0})$, so $\cdots\to Q_{2}\to Q_{1}\oplus Q'_{0} \to l(P_{0})\to A\to 0$ is exact. Continuing the process one obtains the assertion.
\end{proof}

Now we are in a position to state the main result in this subsection.

\begin{theorem} \label{thm:3.8}
Let $(\mathcal{B},\mathcal{A}, i,e,l)$ be a cleft extension with $i$ preserving coproduct,
and $l(P_{1})\xrightarrow{\delta}l(P_{0})\to A\to 0$ be a projective presentation of $A\in\A$.
If $A$ is (resp. partial) silting with respect to $\delta$, then $q(A)$ is (resp. partial) silting with respect to $\mathbf{q}(\delta)$.
\end{theorem}

\begin{proof}
The exact sequence $l(P_{1})\xrightarrow{\delta}l(P_{0})\to A\to 0$ induces the exact sequence $P_{1}\xrightarrow{q(\delta)}P_{0}\to q(A)\to 0$. 

Suppose that $A$ is partial silting with respect to $\delta$. Let $(Y_{j})_{j\in J}$ be a family of objects in $\B$ with $Y_{j}\in\mathcal{D}_{q(\delta)}$. Then $i(Y_{j})\in\mathcal{D}_{\delta}$ by Lemma \ref{lem:3.6} and so $i(\coprod_{j\in J}Y_{j})\cong \coprod_{j\in J}i(Y_{j})\in\mathcal{D}_{\delta}$. Thus $\coprod_{j\in J}Y_{j}\in\mathcal{D}_{q(\delta)}$ by Lemma \ref{lem:3.6} again. So $\mathcal{D}_{q(\delta)}$ is a torsion class.

Let $f\in\Hom_{R}(P_{1}, q(A))$. From (2.1) we have that $le(A)\xrightarrow{\mu_{A}} A\to 0$ is exact,
and hence $e(A)\xrightarrow{q(\mu_{A})} q(A)\to 0$ is exact as $q$ is right exact.
Then there is $g: P_{1}\to e(A)$ such that $f=q(\mu_{A})g$.
Since $(l,e)$ is an adjoint pair, we get the following commutative diagram
\[
\begin{array}{ccc}
\Hom_{\A}(l(P_{0}), A) & \xrightarrow{\operatorname{Hom}_{\A}(\delta, A)} & \operatorname{Hom}_{\A}(l(P_{1}), A) \\
\cong \Bigg\downarrow & & \Bigg\downarrow \cong  \\
\operatorname{Hom}_{\B}(P_{0}, e(A)) & \xrightarrow{\operatorname{Hom}_{\B}(q(\delta), e(A))} & \operatorname{Hom}_{\B}(P_{1}, e(A)) 
\end{array}
\]

 Since $A\in\mathcal{D}_{\delta}$, then $\Hom_{\A}(\delta, A)$ is epic,
and hence there is $h: P_{0}\to e(A)$ such that $g=hq(\delta)$. 
 So $f=q(\mu_{A})hq(\delta)$. Hence $q(A)\in\mathcal{D}_{q(\delta)}$. Thus $q(A)$ is partial silting with respect to $q(\delta)$.

Now suppose that $A$ is silting with respect to $\delta$. Then $q(A)$ is partial silting with respect to $q(\delta)$. So $\Gen(q(A))\subseteq\mathcal{D}_{q(\delta)}$.

Let $Y\in\mathcal{D}_{q(\delta)}$.  Lemma \ref{lem:3.6} yields $i(Y)\in\mathcal{D}_{\delta}=\Gen A$. Thus there are a cardinal $J$ and an epimorphism $f: A^{(J)}\to i(Y)$, which implies that $q(f):q(A)^{(J)}\to qi(Y)$ is an epimorphism. 
Note that $qi(Y)\overset{\eta_{Y}}\to Y\to 0$ is exact for all $Y\in\B$ by Lemma \ref{lem:2.1}(2),
where $\eta\colon qi\rightarrow \mathsf{Id}_{\mathcal{B}}$ is the counit of the adjoint pair $(q,i)$. 
Hence $\eta_{Y}q(f): q(A)^{(J)}\to Y$ is epic and so $Y\in\Gen(q(A))$. 
So $\mathcal{D}_{q(\delta)}\subseteq\Gen(q(A))$. Thus $q(A)$ is  silting with respect to $q(\delta)$.
\end{proof}

The following result is a special case of the above theorem.

\begin{corollary} \label{cor:3.9}
Let $B$ be an object in $\B$.
If $i$ preserves coproduct and $l(B)$ is (resp. partial) silting with respect to $\delta$, then $B$ is (resp. partial) silting with respect to $\mathbf{q}(\delta)$.
\end{corollary}

\begin{remark} \label{rmk:3.10}
{\rm Let $(\mathcal{B},\mathcal{A},i,e,l)$ be a cleft extension. If $e$ has a right adjoint $r$,
then it is also a cleft coextension, and hece $i$ and $e$ both preserve coproduct. 
A nice situation occurs when a cleft extension is also a cleft coextension. For example, we have the following result:}
\end{remark} 

\begin{corollary} \label{thm:3.11}
Let $(\mathcal{B},\mathcal{A},i,e,l,r)$ be a cleft extension with $e$ having a right adjoint functor $r$. Then
\begin{enumerate}
\item Let $P_{1}\xrightarrow{\sigma}P_{0}\to B\to 0$ be a projective presentation of $B$ in $\B$.
 $l(B)\in\A$ is silting with respect to $l(\sigma)$ if and only if $B$ is so with respect to $\sigma$ and $F(B)\in$ $\Gen(B)$.
\item Let $0\to B\to E^{0}\xrightarrow{\xi}E^{1}$ be an injective copresentation of $B$. $r(B)\in\A$ is cosilting with respect to $r(\xi)$ if and only if 
$B$ is so with respect to $\xi$ and $F'(B)\in\Cogen(B)$.
\end{enumerate}
\end{corollary}

\begin{proof}
The statements are followed from Theorem \ref{thm:3.3} and its dual.
\end{proof}

\section{Applications}

In this section, we will give some applications to $\theta$-extensions and tensor rings, 
which recover some known results about trivial ring extensions and triangular matrix rings.
In what follows, all rings are nonzero associative rings with identity and all modules are unitary. For a ring $R$, we write $\Mod R$ for the category of left $R$-modules. 

\begin{definition}{\rm (\cite{M93})}
{\rm Let $R$ be a ring, $M$ an $R$-$R$-bimodule and $\theta : M \otimes _{R} M \rightarrow M $ an
associative $R$-bimodule homomorphism. The {\it $\theta$-extension} of $R$ by $M$, denoted by
$R\ltimes _{\theta} M$, is defined to be the ring with underlying group $R \oplus M$ and multiplication
given as follows: $$(r, m)\cdot(r',m'):=(r r',r m'+mr'+\theta (m\otimes m')),$$
for any $r, r'
\in R$ and $m, m' \in M$.}
\end{definition}
Let $T:=R\ltimes _{\theta} M$.
Then a left $T$-module is identified with a pair
$(X,\alpha)$, where $X\in \Mod R$ and $\alpha \in \Hom _R(M\otimes_{R}X, X)$ such that $\alpha \circ (1_M\otimes \alpha)
=\alpha\circ(\theta \otimes 1_X)$.
Moreover, we have the following ring homomorphisms
$R\rightarrow T$ given by $r \mapsto (r, 0)$ and $T\rightarrow R$ given by $(r, m) \mapsto r$.
By \cite[Section 6.3]{KP25}, we have the following
cleft extension
of module categories
$$\xymatrix@!=8pc{ \Mod R \ar[r]|{i=_TR\otimes _R-} & \Mod T
			 \ar[r]|{e=_RT\otimes _T-} \ar@/_2pc/[l]|{q=_RR\otimes _T-} & \Mod R
			\ar@/_2pc/[l]|{l=_TT\otimes _R-} \ar@(ur,ul)_F    },  $$
where $q(X,\alpha)=\Coker \alpha$,
$i(Y)=(Y,0)$, $l(Y)=(Y\oplus M\otimes _R Y,\tiny {\left(\begin{array}{cc} 0 &0 \\ 1 & \theta \otimes 1_Y \end{array}\right)})$,
$e(X,\alpha)=X$ and $F(Y)=M\otimes _RY$.
Moreover, every cleft extension of module categories
is isomorphic to a cleft extension induced by a
$\theta$-extension, see \cite[Proposition 6.9]{KP25}.
It is easy to see that $e$ has a right adjoint functor $r=\Hom_{R}(T, -)$,
hence $(\Mod R,\Mod T, i,e,l)$ is also a cleft coextension.

Applying the foregoing results to $\theta$-extension,
we have the following two conclusions, which are generalizations of \cite[Theorem 3.3 and Corollary 3.5]{M22} and \cite[Theorem 4.3]{M22}, respectively,
since a trivial ring extension $R\ltimes M$ is a $\theta$-extension with $\theta=0$.


\begin{proposition} \label{prop:4.2}
\begin{enumerate}
\item Let $P_{1}\xrightarrow{\sigma}P_{0}\to Y\to 0$ be a projective presentation of a left $R$-module $Y$. Then $l(Y)\in\Mod T$ is (resp. partial) silting with respect to $l(\sigma)$ iff
$Y$ is so with respect to $\sigma$ and $M\otimes_{R}Y\in$ (resp. $\mathcal{D}_{\sigma})$ $\Gen(Y)$.
\item If $M$ is $m$-nilpotent for some $m$ and $\Tor_{t}^{R}(M, Y)=0$ for all $1\leq t\leq n+1$, 
then $l(Y)$ is (resp. partial) $n$-tilting iff $Y$ is so and $M\otimes_{R}Y\in Y^{\bot_{n}}$.
\end{enumerate}
\end{proposition}

\begin{proof}
Since the cleft extension $(\Mod R,\Mod T, i,e,l)$ is also a cleft coextension, 
the statement (1) is a direct corollary of Theorem \ref{thm:3.3} and Remark \ref{rmk:3.10}.

(2) Since $M$ is $m$-nilpotent, so is $F$. This implies that  the natural transformation $\eta: F^{2}\to F$ is $m$-nilpotent. From \cite[Lemma 2.3]{M93} we know that $\Ker q=0$.
Thus the assertion is obtained from Theorem \ref{thm:3.5}.
\end{proof}

The second result follows from Theorem \ref{thm:3.8} and Remark \ref{rmk:3.10} immediately.

\begin{proposition} \label{prop:4.3}
Let $l(P_{1})\xrightarrow{\delta}l(P_{0})\to (X,\alpha)\to 0$ be a projective presentation of $(X,\alpha)\in\Mod T$. If $(X,\alpha)$ is (resp. partial) silting with respect to $\delta$, then $\Coker(\alpha)$ is a (resp. partial) silting left $R$-module with respect to $q(\delta)$.
\end{proposition}

Let $R$ be a ring and $N$ a $R$-$R$-bimodule.
Recall that the {\it tensor ring} $T_R(N) =\oplus _{i=0}^{\infty}N^i$,
where $N^0 = R$ and $N^{i+1} = N \otimes _R N^{i}$ for $i \geq 0$. It follows from
\cite{KP25} that there is an isomorphism $T_R(N)\cong R\!\ltimes\!_{\theta}M'$ where $M'=N\oplus N^{\otimes 2}\oplus \cdots$ and $\theta$ is induced by $N^{\otimes k}\otimes_{\Gamma}N^{\otimes l}\rightarrow N^{\otimes (k+l)}$. Evidently, $M'$ is nilpotent if the same holds for $N$. Thus tensor ring is a special $\theta$-extension with $\theta \neq 0$. 

Applying two Propositions above to tensor rings, we can establish connections between silting and tilting modules over a tensor ring $T_R(N)$ and those over a ring $R$.

\begin{corollary} \label{cor:4.4}
Assume that $N$ is $m$-nilpotent and $P_{1}\xrightarrow{\sigma}P_{0}\to Y\to 0$ be a projective presentation of a left $R$-module $Y$.
\begin{enumerate}
\item $l(Y)$ is a  (resp. partial) silting left $T_R(N)$-module with respect to $l(\sigma)$ iff
$Y$ is so with respect to $\sigma$ and $N^{\otimes_R i}\otimes_{R}Y\in$ (resp. $\mathcal{D}_{\sigma})$ $\Gen(Y)$ for $1\leqslant i\leqslant m$.
\item  For all $1\leqslant i\leqslant m$, if $\Tor_{t}^{R}(N^{\otimes_R i}, Y)=0$ for all $1\leq t\leq n+1$, 
then $l(Y)$ is (resp. partial) $n$-tilting iff $Y$ is so and $N^{\otimes_R i}\otimes_{R}Y\in Y^{\bot_{n}}$.
\end{enumerate}
\end{corollary}

\begin{corollary} \label{cor:4.5}
Let $l(P_{1})\xrightarrow{\delta}l(P_{0})\to (X,\alpha)\to 0$ be a projective presentation of $(X,\alpha)\in\Mod T_R(N)$.
If $(X,\alpha)$ is (resp. partial) silting with respect to $\delta$, then $\Coker(\alpha)$ is a (resp. partial) silting left $R$-module with respect to $q(\delta)$.
\end{corollary}

Let $R$ be a finite dimensional algebra over an algebraically closed field.  We denote by $\tau$ the Auslander-Reiten translation in the category mod $R$ of finitely generated left $R$-modules, and denote by $|Y|$ the number of non-isomorphic indecomposable direct summands of $Y\in$ mod $R$. Recall from \cite{AIR14} that $Y\in$ mod $R$ is called {\it $\tau$-rigid}
if $\Hom_{R}(Y, \tau Y)=0$, and called  {\it $\tau$-tilting} if it is $\tau$-rigid and $|Y| = |R|$.
$Y$ is said to be  {\it support $\tau$-tilting} if there is an idempotent element $a$ of $R$ such that $Y$ is a $\tau$-tilting $R/RaR$-module.

The following result provides a method for constructing support $\tau$-tilting $R\ltimes _{\theta} M$-modules from the ones of $R$-modules, which extends the main result of \cite{LZ22} and \cite{GH20}.

\begin{corollary} \label{cor:4.6}
Let $T$ be a $\theta$-extension of $R$ by a finitely generated bimodule $M$, and $Y\in$ mod $R$. Then 
$l(Y)$ is a (resp. $\tau$-rigid) support $\tau$-tilting left $T$-module if and only if 
$Y$ is so and $\Hom_{R}(M\otimes_{R}Y, \tau Y)=0$.
\end{corollary}

\begin{proof}
Suppose that $P_{1}\xrightarrow{\sigma}P_{0}\to Y\to 0$ is a minimal projective presentation of $Y$. Then
$M\otimes_{R}Y\in\mathcal{D}_{\sigma}$ if and only if $\Hom_{R}(M\otimes_{R}Y, \tau Y)=0$ by \cite[Proposition 2.4]{AIR14}.
From \cite[Proposition 3.15]{AMV16} we know that $Y$ is (resp. $\tau$-rigid) support $\tau$-tilting  if and only if it is (resp. partial) silting. The assertion is obtained from Proposition \ref{prop:4.2}(1).
\end{proof}

The first statement in the next result generalizes\cite[Proposition 3.3]{LZ22}, and the second one is a new result.

\begin{corollary} \label{cor:4.7}
Let $T$ be a $\theta$-extension of $R$ by a finitely generated bimodule $M$, and $(X,\alpha)$ be a left $R\ltimes_{\theta}M$-module.
\begin{enumerate} 
\item If $(X,\alpha)$ is $\tau$-rigid, then $\Coker(\alpha)$ is a $\tau$-rigid left $R$-module.
\item If $(X,\alpha)$ is support $\tau$-tilting, then $\Coker(\alpha)$ is a support $\tau$-tilting left $R$-module.
\end{enumerate}
\end{corollary}

\begin{proof}
It follows from Proposition \ref{prop:4.3} and \cite[Proposition 3.15]{AMV16}.
\end{proof}

\section*{Acknowledgments}

This research was partially supported by the National Natural Science Foundation of China (Grant No. 12061026) and the Key Research Projects Plan of Higher Education Institutions in Henan Province (Grant No. 26B110007).


\end{document}